\documentclass[reqno,A4paper]{amsart}
\usepackage{amsmath}
\usepackage{amssymb}
\usepackage{amsthm}
\usepackage{enumerate}
\usepackage{mathrsfs}
\usepackage{array}
\newtheorem{claim}[equation]{Claim}
\newtheorem{definition}[equation]{\indent{\it Definition}\rm }

\newtheorem{lemma}[equation]{Lemma}

\newtheorem{theorem}[equation]{Theorem}

\newcommand{\C}{{\mathbb{C}}}
\newcommand{\A}{{\mathbb{A}}}

\newcommand{\ord}{{\mathrm{ord}}}

\renewcommand{\div}{{\mathop{\mathrm{div}}}}
\newcommand{\ric}{{\mathrm{Ric}\,}}

\newcommand{\St}{\mathrm{St}}

\newcommand{\supp}{\mathrm{Supp}\,}

\numberwithin{equation}{section}

\begin{document}

\title[Holomorphic curves from annuli into semi Abelian varieties]{Nevanlinna theory for holomorphic curves from annuli into semi Abelian varieties} 

\author[Si Duc Quang]{Si Duc Quang$^{1,2}$}
\newcommand{\acr}{\newline\indent}
\address{\llap{$^1$\,}Department of Mathematics\acr
Hanoi National University of Education\acr
136-Xuan Thuy - Cau Giay\acr
Hanoi\acr
VIETNAM}

\address{\llap{$^2$\,}Thang Long Instutute of Mathematics and Applied Sciences\acr
Nghiem Xuan Yem, Hoang Mai\acr
Hanoi\acr
VIETNAM}

\email{quangsd@hnue.edu.vn}

\thanks{This work was done during a stay of the author at Vietnam Institute for Advanced Study in Mathematics (VIASM). He would like to thank the institute for the support. }


\begin{abstract} 
In this paper, we prove a lemma on logarithmic derivative for holomorphic curves from annuli into K\"{a}hler compact manifolds. As its application, a second main theorem for holomophic curves from annuli into semi abelian varieties intersecting with only one divisor is given.
\end{abstract}

\maketitle

\section{Introduction}

Let $f$ be an algebraically non-degenerate holomorphic curve from $\C$ into a semi-Abelian variety $M$ and let $D$ be an algebraic divisor on $M$. In 2002, J. Noguchi, J. Winkelmann and K. Yamanoi \cite{NWY} proved that there exist a good compactification $\overline{M}$ of $M$ and an integer $k_0=k_0(f, D)$ satisfying
\begin{align*}
T_f(r;c_1(\overline D))\le N^{[k_0]}(r, f^*D)+O(\log T_f(r; c_1(\overline{D})))+O(\log r)
\end{align*}
for all $r\in [0;+\infty)$ outside a finite Borel measure set. Here by $T_f(r;c_1(\overline D))$ and $N^{[k_0]}(r, f^*D)$ we denote the characteristic function of $f$ with respect to the line bundle $L(\overline{D})$ in $\overline{M}$ and the counting function of divisor $f^*D$ truncated to level $k_0$ (see Section $\S 2$ for the definitions).

Adapting the method of the above three authors and using the lemma on logarithmic derivative given by Noguchi \cite{Nog81}, recently Quang \cite{Q12} has generalized the above result to the case of holomorphic curves from punctured disc $\Delta^*=\{z\in\C\ :\ |z|\ge 1\}$ into a semi-Abelian variety $M$. Also in \cite{Q12}, as an application of his second main theorem,  Quang gave an alternative proof of Big Picard's theorem for algebraically non degenerate mappings $f: \Delta^* \to M \setminus D$.

In this paper, we will extend these above results to the case of holomorphic curves from annuli into semi-Abelian varieties. In order to establish the second main theorem, we firstly prove a lemma on logarithmic derivative for holomorphic curves from annuli into K\"{a}hle manifold. To state our results, we recall the following.

For $R_0>1$, we set the annulus
$$ A(R_0)=\left\{z\in\C\ :\ \frac{1}{R_0}<|z|<R_0\right\}.$$
Let $N$ be a compact K\"{a}hler manifold. Let $\mathcal M_N^{*}$ be the sheaf of germs of meromorphic functions on $N$ which do not identically vanish, and define a sheaf $\mathcal U^{1}_{N}$ by
$$\begin{array}{ccccc}
0\rightarrow \C^{*}\rightarrow &\mathcal M^{*}_{N}&\overset{d\log}{\rightarrow}&\mathcal U^{1}_{N} & \rightarrow 0.\\
&\gamma&\mapsto &d\log\gamma & \\
\end{array},$$
where $\C^*$ denotes the multiplicative group of non-zero complex numbers.

Our lemma on logarithmic derivative  is stated as follows.
\begin{lemma}\label{1.1}
Let $N$ be a compact K\"{a}hler manifold with K\"{a}hler metric $h$ and the associated form $\Omega$. Let $f:\A(R_0)\rightarrow N$ be a holomorphic curve from annulus $\A(R_0)\ (R_0>1)$ into $N$ and let $\omega\in H^{0}(N, \mathcal U^1_N)$. Setting $f^{*}\omega =\xi (z)dz$, we have
$$ \bigl\|\ m_0(r,\xi )\le O\left(\log^+ T_{0,f}(r, \Omega )\right)+O\left(\log^+\frac{1}{R_0-r}\right).$$
\end{lemma}
Here, $m_0(r,\xi )$ denotes the proximitive function of $\xi$ (see Section $\S 2$ for the definition) and the notation $\|\ P$ means the assertion $P$ holds for all $r\in [1;R_0)$ outside a Borel subset $E$ with $\int\limits_{E}\frac{dt}{(R_0-r)^{\lambda +1}}\le +\infty$ for some positive number $\lambda$. We learn the technique of the proof of Lemma \ref{1.1} from \cite{Nog81}.

Our second main theorem in this paper is stated as follows.

\begin{theorem}\label{1.2}
Let $f:\A(R_0)\rightarrow M$ be an algebraically non-degenerate holomorphic curve into a semi-Abelian variety $M$ and let $D$ be a reduced divisor on $M$. Then there exist a smooth equivariant compactification $\overline M$ of $M$ independent of $f$ and a natural number $k_0$ such that
\begin{align*}
\bigl\|\ T_{0,f}(r; c_1(\overline D)) &= N_0^{[k_0]}(r, f^*D)+O\left (\log^+T_{0,f}(r; c_1(\overline D))+\log^+\frac{1}{R_0-r}\right).
\end{align*}
\end{theorem}

The basic notation of this paper is from \cite{KK05a,KK05b,LY10,Nog81,NWY} and \cite{Q12}.

\section{Basic Notion from Nevanlinna theory and Semi Abelian varieties}

\vskip0.2cm
\noindent
\textbf{(a)} Meromorphic functions on annuli.

Let $R_0>1$ and let $\A(R_0)$ be an annulus. For a divisor $\nu$ on $\A (R_0)$, which we may regard as a function on  $\A (R_0)$ with values in $\mathbb Z$ whose support is a discrete subset of $\A (R_0),$ and  for a positive integer $k$ (maybe $k= + \infty$), the \textit{counting function} of $\nu$ is defined by
\begin{align*}
n_0^{[k]}(t)&=\begin{cases}
\sum\limits_{1\le |z|\le t}\min\{k,\nu (z)\}&\text{ if }1\le t<R_0\\
\sum\limits_{t\le |z|<1}\min\{k,\nu (z)\}&\text{ if }\dfrac{1}{R_0}<t< 1
\end{cases}\\
 \text{ and }N_0^{[k]}(r,\nu)&=\int\limits_{\frac{1}{r}}^1 \dfrac {n_0^{[k]}(t)}{t}dt +\int\limits_1^r \dfrac {n_0^{[k]}(t)}{t}dt \quad (1<r<R_0).
\end{align*}

Let $f$ be a meromorphic function on $\A(R_0)$. We define the \textit{proximity function} by
$$ m_0(r,f)=\dfrac{1}{2\pi}\int\limits_{0}^{2\pi}\log^+\left|f\left(\dfrac{1}{r}e^{i\theta}\right)\right| d\theta +\dfrac{1}{2\pi}\int\limits_{0}^{2\pi}\log^+|f(re^{i\theta})| d\theta- \dfrac{1}{\pi}\int\limits_{0}^{2\pi}\log^+|f(e^{i\theta})| d\theta. $$
The \textit{characteristic function} of $f$ is defined by
$$ T_0(r,f)=m_0(r,f)+N_0(r,\nu^\infty_f). $$
We note that these definition also available for multiplicative meromorphic functions.

The function $f$ is said to be admissible if it satisfies
$$ \underset{r\longrightarrow R_0^-}{\mathrm{limsup}}\dfrac{T_0(r,f)}{-\log (R_0-r)}=+\infty.$$
Throughout this paper, a Borel subset $E$ of $[1;R_0)$ is said to be an $\Delta_{R_0}$-set if it satisfies
$$\int_{E}\dfrac{dr}{(R_0-r)^{\lambda +1}} <+\infty$$ 
for some $\lambda \ge 0$. 

\vskip0.2cm
\noindent
\textbf{(b)} Holomorphic curves from annuli into K\"{a}hler compact manifolds.

Let $\xi$ be a function on $\A(R_0)$ satisfying that

\hskip0.3cm (i) $\xi$ is differentiable outside a discrete set of points,

\hskip0.3cm (ii) $\xi$ is locally written as a difference of two subharmonic  functions.

\noindent
 Then by \cite[$\S 1$]{Nog81}, we easily have
\begin{align}\label{2.1}
\begin{split}
\int_{1}^{t}\dfrac{dt}{t}\int_{\A(t)}dd^c\xi
 =&\dfrac{1}{4\pi}\int_{|z|=r}\xi (z)d\theta +\dfrac{1}{4\pi}\int_{|z|=\frac{1}{r}}\xi (z)d\theta \\
& -\dfrac{1}{2\pi}\int_{|z|=1}\xi (re^{i\theta})d\theta -2(\log
 r)\int_{|z|=1}d^c\xi,
\end{split}
 \end{align}
where $dd^c\xi$ is taken in the sense of current.

Let $N$ be a K\"{a}hler compact manifold with K\"{a}hler metric $h$ and the associated form $\Omega$.
Let $f:\A(R_0)\rightarrow N$ be a holomorphic curve. The {\it characteristic function} of $f$ with respect to $\Omega$ is defined by
\begin{align}\label{2.2}
T_{0,f}(r;\Omega)=\int_{1}^{r}\dfrac{dt}{t}\int_{\A(t)}f^*\Omega, \ \ 1<r<R_0.
\end{align} 

Let $D$ be an effective divisor on $N$. We assume that $f(\A(R_0))\not\subset D$. We denote $L(D)$ the line bundle determined by $D$. We fix a Hermitian fiber metric $\|\cdot \|$ in $L(D)$ with the curvature form $\omega$ representing the first Chern class $c_1(D)$ of $L(D)$. Take $\sigma\in H^0(N,L(D))$ with
$\div (\sigma)=D$ and $\|\sigma \|\le 1$ (by the compactness of $N$). We set
\begin{align}\label{2.3}
T_{0,f}(r; c_1(D)):=T_{0,f}(r;\omega )=\int_{1}^{r}\dfrac{dt}{t}\int_{\A(t)}f^*\omega, \ \ 1<r<R_0,
\end{align}
which is well-defined up to an $O(1)$-term. 
The {\it proximity function} of $f$ with respect to $D$ is
defined by
\begin{align}\label{2.4}
\begin{split}
m_{0,f}(r;D)=&\dfrac{1}{2\pi}\int_{|z|=r}\log\dfrac{1}{\|\sigma (f(z))\|}d\theta +\dfrac{1}{2\pi}\int_{|z|=\frac{1}{r}}\log\dfrac{1}{\|\sigma (f(z))\|}d\theta\\
&-\dfrac{1}{\pi}\int_{|z|=1}\log\dfrac{1}{\|\sigma (f(z))\|}d\theta.
\end{split}
\end{align}
Applying (\ref{2.1}) to $\xi=f^*\log\|\sigma \|$, 
we obtain the {\bf First Main Theorem}:
\begin{align}\label{2.5}
T_{0,f}(r; c_1(D))=N_0(r,f^*D) &+m_{0,f}(r;D)+O(1).
\end{align}


\noindent
\textbf{(c)} Semi Abelian varieties and Logarithmic Jet bundle

Let $M_0$ be an Abelian variety and let $M$ be a complex Lie group admitting the exact sequence 
\begin{align}\label{2.6}
0\rightarrow\ (\C^*)^p\rightarrow M\overset{\eta}{\rightarrow} M_0\rightarrow 0,
\end{align}
where $\C^*$ is the multiplicative group of non zero complex numbers.
Such an $M$ is called a \textit{semi Abelian variety}.

Taking the universal covering of (\ref{2.6}), one gets
$$\begin{array}{ccccccccc}
0&\rightarrow & \C^p&\rightarrow &\C^n&\rightarrow &\C^m&\rightarrow &0\\
 &&\downarrow &&\downarrow &&\downarrow &&\\
0&\rightarrow &(\C^*)^p &\rightarrow &M&\rightarrow &M_0&\rightarrow &0,\\
\end{array}$$
\noindent
and an additive discrete subgroup $\Lambda$ of $\C^n$ such that
\begin{align*}
\pi :\C^n&\rightarrow M=\C^n/\Lambda, \\
\pi_0:\C^m=(\C^n/\C^p)&\rightarrow M_0 =(\C^n/\C^p)/(\Lambda /\C^p), \\
(\C^*)^p&=\C^p/(\Lambda\cap\C^p).
\end{align*}

Take a smooth equivariant compactification $\overline M$ of $M$. Then the boundary divisor $\partial M$ has only simple normal crossings.
Denote by $\Omega^1_{\overline{M}}(\log \partial M)$ the sheaf of germs of logarithmic $1$-forms over $\overline{M}$.
We take a basis $\{\omega^j\}_{j=1}^n$ of $H^0(\overline{M}, \Omega^1_{\overline{M}}(\log \partial M))$, which are $d$-closed, invariant with respect to the action of $M$, and $(\omega^1 \wedge \cdots \wedge \omega^n)(x)\not=0$ at all $x \in M$.
By the pairing $\{\omega^j\}_{j=1}^n$, we get the
 following trivialization of the logarithmic tangent bundle:
\begin{equation}\label{2.7}
T(\overline M,\log \partial M) \cong \overline{M} \times \C^n.
\end{equation}
Moreover, we have the logarithmic $k$-jet bundle $J_k(\overline M, \log \partial M)$ over $\overline M$ and a natural morphism
$$ \psi: J_k(\overline M, \log \partial M) \to J_k(\overline{M}).$$
The trivialization \eqref{2.7} gives
$$J_k(\overline M, \log \partial M)\cong\overline M\times \C^{nk}.$$
Let
\begin{align*}
\pi_1 &:J_k(\overline M, \log \partial M)\cong\overline M\times \C^{nk}\rightarrow\overline M,\\
\pi_2 &:J_k(\overline M, \log \partial M)\cong\overline M\times \C^{nk}\rightarrow\C^{nk}
\end{align*}
be the projections.
 For a $k$-jet $y\in J_k(\overline M, \log \partial M)$
 we call $\pi_2(y)$ the jet part of $y$.

Let $x\in\overline D$ and let $\sigma =0$ be a local defining equation of $\overline D$ around $x$. For a germ $g:(\C ,0)\rightarrow (M,0)$ of holomorphic mappings, we denote its $k$-jet by $j_k(g)$, and write
$$d^j\sigma (g)={\dfrac{d^j}{d\xi^j}\biggl |}_{\xi =0}\sigma (g(\xi )).$$
We set
\begin{align*}
&J_k(\overline D)_x = \{j_k(g)\in J_k(\overline M)_x \ |\ d^\sigma (g)=0,\ 1\le j\le k\},\\
&J_k(\overline D) =\cup_{x\in\overline D}J_k(\overline D)_x,\\
&J_k(\overline D, \log \partial M) =\psi^{-1}J_k(\overline D).
 \end{align*}
 $J_k(\overline M, \log \partial M)$, which is depending in general on
 the embedding $\overline D\hookrightarrow \overline M$(cf.\
 \cite{Nog86}).
 Note that $\pi_2\big (J_k(\overline D, \log \partial M)\big )$ is an
 algebraic subset of $\C^{nk}$, since $\pi_2$ is proper.

\vskip0.2cm
\noindent
\textbf{(d)} Divisor of semi Abelian variety in general position

Let $M$ be the semi-Abelian variety as above and let $X$ be a complex algebraic variety, on which $M$ acts:
$$(a,x)\in M\times X\rightarrow a \cdot x\in X$$
Let $Y$ be a subvariety embedded into a Zariski open subset of $X$.

\begin{definition}[{see \cite[Definition 3.2]{NWY}}]
We say that $Y$ is generally positioned in $X$ if the closure $\overline Y$ of $Y$ in $X$ contains no $M$-orbit. 
If the support of a divisor $E$ on a Zariski open subset of $X$ is generally positioned in $X$, then $E$ is said to be generally positioned in $X$.
\end{definition}

\begin{definition}
 Let $Z$ be a subset of $M$. We define the {\it stabilizer} of $Z$ by
$$\mathrm{St}(Z)=\{x\in M|x+Z=Z\}^0,$$
 where $\{\cdot\}^0$ denotes the identity component.
\end{definition}

\section{Proof of Lemma on logarithmic derivative}

In this section, we will give the proof for Lemma \ref{1.1}. 


The following is a general property of the characteristic function ( see \cite[Lemma 6.1.5]{NO} for reference).

\begin{theorem}[{see \cite[Lemma 6.1.5]{NO}}]\label{3.2}
Let $f : \A(R_0) \rightarrow V$  be a holomorphic curve into a complex projective 
manifold $V$ and let $H$ be a big line bundle on $V$. 
Then
$$T_{0,f}(r; c_1(L))=O(T_{0,f}(r;c_1(H)))+O\left(\log \frac{1}{R_0-r}\right),$$
for every line bundle $L$ on $V$.
\end{theorem}

\begin{lemma}\label{3.3}
Let $\varphi$ be a positive monotone increasing function in $r\in [1;R_0)\ (R_0>1)$
Then for every $\lambda >0$, we have
$$\biggl\|\ \frac{d}{dr}(\varphi)\le \left (\frac{\varphi}{R_0-r}\right)^{\lambda +1}.$$ 
\end{lemma}
\begin{proof}
Let $E=\left\{r\in [1;R_0)\ :\ \frac{d}{dr}(\varphi)> \left (\frac{\varphi}{R_0-r}\right)^{\lambda +1}\right\}$. Since $\varphi$ is a monotone increasing function, its derivative $\frac{d}{dr}(\varphi)$ exists almost everywhere. Hence $E$ is a Borel measurable subset of  $[1;R_0)$. Then we have
$$ \int\limits_E\frac{dt}{(R_0-r)^{\lambda +1}}\le \int\limits_E\frac{\varphi'}{\varphi^{\lambda +1}}dt\le \frac{1}{\lambda}\left (\frac{1}{\varphi^\lambda (1)}-\lim\limits_{r\rightarrow R_0^-}\frac{1}{\varphi^\lambda (r)}\right )=O(1).$$
The lemma is proved.
\end{proof}

\begin{lemma}[{see \cite[Lemma 2.12]{Nog81}}]
\label{3.1}
Let $f$ be a nonzero multiplicative meromorphic function on $\A (R_0)$. Then for each positive integer $k$ and positive number $\epsilon$, we have
$$\biggl\|\ m_0\left (r,\dfrac{f'}{f}\right)\le (4+\epsilon)\log^+T_0(r,f)+O\left(\log^+\frac{1}{R_0-r}\right). $$
\end{lemma}

\begin{proof}
We denote by $\omega$ the standard complex coordinate on $\overline{\C}=\C\cup{\infty}$ and consider the canonical K\"{a}hler form
$$ \Psi_0=\frac{1}{(1+|\omega|^2)^2}\frac{i}{2\pi}d\omega\wedge d\bar\omega.$$
By Griffiths-King \cite[Proposition 6.9]{GK}, we may choose suitably positive constants $a,b$ and $\delta\ (\delta <1)$ such that the form
$$ \Psi=\frac{a(|\omega|+|\omega|^{-1})^{2+2\delta}}{(\log b(1+|\omega|^2))^2(\log b(1+|\omega|^{-2}))^2}\Psi_0 $$ 
satisfies
$$ \ric\Psi\ge (|\omega|+|\omega|^{-1})^{-2\delta}\Psi. $$
Since $f$ is multiplicative, it is easy to see that $f^*\Psi$ is well-defined. We set
$$ 
\begin{cases}
&g=\dfrac{f'}{f},\ \zeta=\dfrac{a(|f|+|f|^{-1})^{2\delta}}{(\log b(1+|f|^2))^2(\log b(1+|f|^{-2}))^2}|g|^2,\\
&f^*\Psi=\dfrac{a(|f|+|f|^{-1})^{2\delta}}{(\log b(1+|f|^2))^2(\log b(1+|f|^{-2}))^2}|g|^2\dfrac{i}{2\pi}dz\wedge d\bar z=\zeta\dfrac{i}{2\pi}dz\wedge d\bar z.
\end{cases}
$$
On the other hand, we have
$$ g^*\ric\Psi=dd^c\log\zeta\ge (|f|+|f|^{-1})^{-2\delta}\frac{i}{2\pi}dz\wedge d\bar z,$$
and 
$$ dd^c\log\zeta=f^*\ric\Psi -\delta\left([\div^0(f)]+[\div^\infty(f)]\right)+[\div^0(g)]-[\div^\infty(g)] $$
in the sense of currents.
By the definition, we have 
$$[\div^{\infty}(g)]=\left[\supp(\div^0(f)+\div^\infty(f))\right]\ge [\div^0(f)]+[\div^\infty(f)],$$ and hence we deduce
$$(|f|+|f|^{-1})^{-2\delta}\frac{i}{2\pi}dz\wedge d\bar z\le (1+\delta)\left([\div^0(f)]+[\div^\infty(f)]\right)+dd^c\log\zeta.  $$
Then, by the formula (\ref{2.1}), we have
\begin{align*}\int\limits_1^t&\frac{dt}{t}\int\limits_{\A (t)}\frac{\zeta}{(|f|+|f|^{-1})^{2\delta}}\frac{i}{2\pi}dz\wedge d\bar z\le (1+\delta)\left(N_0(r,\div^0(f))+N_0(r,\div^0(f))\right)\\
&+\frac{1}{4\pi}\int\limits_{|z|=r}\log\zeta d\theta+\frac{1}{4\pi}\int\limits_{|z|=1/r}\log\zeta d\theta -2\log\zeta\int\limits_{|z|=1}d^c\log\zeta -\frac{1}{2\pi}\log\zeta d\theta.
\end{align*}
By the definition of $\zeta$, we have
$$ \frac{1}{4\pi}\int\limits_{|z|=r}\log\zeta d\theta+\frac{1}{4\pi}\int\limits_{|z|=1/r}\log\zeta d\theta \le m_0(r,g)+\delta \left (m_0(r,f)+m_0\left(r,\frac{1}{f}\right)\right)+O(1). $$
Therefore, we get
$$ \int\limits_1^t\frac{dt}{t}\int\limits_{\A (t)}\frac{\zeta}{(|f|+|f|^{-1})^{2\delta}}\frac{i}{2\pi}dz\wedge d\bar z\le m_0(r,g)+2(1+\delta)T_0(r,f)+O(1). $$
For simplicity, we set $\Gamma=\{|z|=r\}\cup\{|z|=1/r\}$. We now have the following estimate
\begin{align*}
\|\ m_0(r,g)&=\frac{1}{4\pi}\int\limits_{\Gamma (r)}\log^+\bigl ((\zeta(|f|+|f|^{-1}))^{-2\delta} (\log (1+|f|^2))^2(\log (1+|f|^{-2}))^2\bigl )d\theta +O(1)\\ 
& \le\frac{1}{4\pi}\int\limits_{\Gamma (r)}\log^+(\zeta(|f|+|f|^{-1}))^{-2\delta} d\theta +\log^{+}\bigl (m_0(r,f)+m_0(r,\frac{1}{f})\bigl)+O(1)\\
&\le\frac{1}{2}\log\biggl (1+\frac{1}{2\pi}\int\limits_{\Gamma (r)}\log^+\bigl(\zeta(|f|+|f|^{-1})\bigl)^{-2\delta} d\theta\biggl)+2\log^{+}T_0(r,f)+O(1)\\
&=\frac{1}{2}\log\biggl (1+\frac{1}{2r}\biggl (\frac{d}{dr}\int\limits_{\A (r)}\frac{\zeta}{(|f|+|f|^{-1})^{2\delta}}\frac{i}{2\pi}dz\wedge d\bar z\biggl )\biggl )+2\log^{+}T_0(r,f)+O(1)\\
&\le \frac{1}{2}\log\biggl (1+\frac{1}{2r}\biggl (\int\limits_{\A (r)}\frac{\zeta}{(|f|+|f|^{-1})^{2\delta}}\frac{i}{2\pi}dz\wedge d\bar z\biggl )^2\biggl )\\
&\hspace{20pt}+2\log^{+}T_0(r,f)+O(\log^+\frac{1}{R_0-r})\\
&=\frac{1}{2}\log\biggl (1+\frac{1}{2r}\biggl (r\frac{d}{dr}\int\limits_{1}^r\int\limits_{\A (t)}\frac{\zeta}{(|f|+|f|^{-1})^{2\delta}}\frac{i}{2\pi}dz\wedge d\bar z\biggl )^2\biggl )\\
&\hspace{20pt}+2\log^{+}T_0(r,f)+O(\log^+\frac{1}{R_0-r})\\
&\le\frac{1}{2}\log\biggl (1+\frac{r}{2}(\int\limits_{1}^r\int\limits_{\A (t)}\frac{\zeta}{(|f|+|f|^{-1})^{2\delta}}\frac{i}{2\pi}dz\wedge d\bar z)^4\biggl )\\
&\hspace{20pt}+2\log^{+}T_0(r,f)+O(\log^+\frac{1}{R_0-r})\\
&\le\frac{1}{2}\log\biggl (1+\frac{r}{2}(m_0(r,g)+2(1+\delta)T_0(r,f))^4\biggl )\\
&\hspace{20pt}+2\log^{+}T_0(r,f)+O(\log^+\frac{1}{R_0-r})\\
&\le 2\log^+m_0(r,g)+4\log^+T_0(r,f)+O(\log^+\frac{1}{R_0-r}).
\end{align*}
We note that, for every non negative function $\psi (r)$ and $\epsilon' >0$, $ \log^+\psi (r)\le  \epsilon' \psi(r)+O(1).$
Then we have
$$\|\ m_0(r,g)\le 2\epsilon' m_0(r,g)+4\log^+T_0(r,f)+O(\log r)+O(1),$$
$$ i.e., \|\ m_0(r,g)\le\frac{4}{1-2\epsilon'}T_0(r,f)+O(\log r)+O(1).  $$
Choosing $\epsilon'=\frac{1}{2}(1-\frac{4}{4+\epsilon}),$ we get
$$ \|\ m_0(r,g)\le (4+\epsilon)T_0(r,f)+O(\log r)+O(1). $$
The lemma is proved.
\end{proof}

\begin{proof}[Proof of Lemma \ref{1.1}]
By Weil \cite[p. 101]{W58} (see also \cite{Nog81}), there is a multiplicative meromorphic function $\varphi$ on $N$ and a holomorphic one form $\omega_1$ on $N$ such that 
$$ \omega=d\log\varphi +\omega_1.$$
We set
$$ f^*\omega_1=\xi_1dz\text{ and }f^*(d\log\varphi)=\xi_2dz, $$
where $\xi_2=\frac{(\varphi\circ f)'}{\varphi\circ f}$.
Then we have
\begin{align}\label{3.4}
m_{0}(r,\xi)\le m_0(r,\xi_1)+m_0(r,\xi_2)+O(1).
\end{align}

Firstly, we are going to estimate $m_0(r,\xi_1)$. From the compactness of $N$, there is a positive constant $C$ such that
$$ |\omega_1(v)|^2\le C.h(v,v)\text{ for all holomorphic tangent vectors }v\in T_N.$$
Setting $f^*\Omega=s(z)\frac{i}{2\pi}dz\wedge d\bar z$, we have
$$ |\xi_1(z)|^2\le C.s(z).$$
Hence, by some simple computations we have
\begin{align*}
m_0(r,\xi_1)\le&\frac{1}{2\pi}\int\limits_{|z|=r}\log(1+|\xi_1|^2)d\theta+\int\limits_{|z|=\frac{1}{r}}\log(1+|\xi_1|^2)d\theta+O(1)\\ 
\le& \frac{1}{4\pi}\int\limits_{|z|=r}\log(1+s(z))d\theta+\frac{1}{4\pi}\int\limits_{|z|=r}\log(1+s(x))d\theta+O(1)\\
\le&\frac{1}{2}\log\left (1+\frac{1}{2\pi}\int\limits_{|z|=r}sd\theta\right )+\frac{1}{2}\log\left (1+\frac{1}{2\pi}\int\limits_{|z|=\frac{1}{r}}sd\theta\right )+O(1)\\
=&\frac{1}{2}\log\left (1+\frac{1}{2\pi r}\dfrac{d}{dr}\int\limits_{1\le |z|<r}tsdt\wedge d\theta\right )\\
&+\frac{1}{2}\log\left (1+\frac{r}{2\pi}\dfrac{d}{dr}\int\limits_{1\ge |z|>\frac{1}{r}}tsdt\wedge d\theta\right )+O(1)\\
=&\frac{1}{2}\log\left (1+\frac{1}{2\pi r}\dfrac{d}{dr}\int\limits_{1\le |z|<r}f^*\Omega\right )+\frac{1}{2}\log\left (1+\frac{r^3}{2\pi}\dfrac{d}{dr}\int\limits_{1\ge |z|>\frac{1}{r}}f^*\Omega\right )+O(1)\\
\le&\frac{1}{2}\log\left (1+\frac{1}{2\pi r}(\frac{1}{R_0-r}\int\limits_{1\le |z|<r}f^*\Omega)^2\right )\\
&+\frac{1}{2}\log\left (1+\frac{r^3}{2\pi}(\frac{1}{R_0-r}\int\limits_{1\ge |z|>\frac{1}{r}}f^*\Omega)^2\right )+O(1),
\end{align*} 
for all $r\in [0;R_0)$ outside an $\Delta_{R_0}$-set. Here the last inequality comes from the fact that $\int\limits_{1\le |z|<r}f^*\Omega$ and $\int\limits_{1\ge |z|>\frac{1}{r}}f^*\Omega$ are both monotone increasing functions in $r\in [1;R_0)$ and Lemma \ref{3.3} is applied.
Moreover, we have
$$\int\limits_{1\le |z|<r}f^*\Omega=r\frac{d}{dr}\int\limits_{1}^r\frac{d}{dt}\int\limits_{1\le |z|<t}f^*\Omega\ \text{ and }\ \int\limits_{1\ge |z|>\frac{1}{r}}f^*\Omega=r\frac{d}{dr}\int\limits_{1}^r\frac{d}{dt}\int\limits_{1\ge |z|>\frac{1}{t}}f^*\Omega. $$
Then we have
\begin{align}\label{3.5}
\begin{split}
m_0(r,\xi_1)&\le \log^{+}\left (\frac{d}{dr}\int\limits_{1}^r\frac{d}{dt}\int\limits_{1\le |z|<t}f^*\Omega\right )+ \log^{+}\left (\frac{d}{dr}\int\limits_{1}^r\frac{d}{dt}\int\limits_{1\ge |z|>\frac{1}{t}}f^*\Omega\right )\\
&\ \ \ +2\log^+\frac{1}{R_0-r}+O(1)\\
&\le 2\log^{+}\left (\frac{d}{dr}\int\limits_{1}^r\frac{d}{dt}\int\limits_{\frac{1}{t}< |z|<t}f^*\Omega\right )+2\log^+\frac{1}{R_0-r}+O(1)\\
&\le 2\log^{+}\left (\frac{d}{dr}T_{0,f}(r;\Omega)\right )+2\log^+\frac{1}{R_0-r}+O(1)\\
&\le 2\log^+\left (\frac{1}{R_0-r}T_{0,f}(r;\Omega)\right )^2+2\log^+\frac{1}{R_0-r}+O(1)\\
&=4\log^+T_{0,f}(r;\Omega)+4\log^+\frac{1}{R_0-r}+O(1).
\end{split}
\end{align}
Here, the fourth inequality holds because of Lemma \ref{3.3}.

Now, we will estimate $m_0(r,\xi_2)$. By Lemma \ref{3.1} we have
\begin{align}\label{3.6}
\biggl\|\ m_0(r,\xi_2)=m_0\left(\frac{(\varphi\circ f)'}{\varphi\circ f}\right)\le O\left(\log^+T_{0}(r,\varphi\circ f)+\log^+\frac{1}{R_0-r}\right).
\end{align}
Take $\| .\|_1$ and $\| .\|_2$ the Hermitian fiber metrics on $L(\div^0(\varphi))$ and $L(\div^\infty(\varphi))$ respectively. We take $\sigma_1\in H^0(N,L(\div^0(\varphi)))$ and $\sigma_2\in H^0\left(N,L(\div^\infty(\varphi))\right)$ so that $\div(\sigma_1)=\div^0(\varphi)$, $\div(\sigma_2)=\div^\infty(\varphi)$ and $|\varphi|=\frac{\|\sigma_1\|_1}{\|\sigma_2\|_2}\le\frac{1}{\|\sigma_2\|_2}$ (because of compactness of $N$, we may suppose that $\|\sigma_1\|_1\le 1$ and $\|\sigma_2\|_2\le 1$). Then we have
\begin{align}\label{3.7}
\begin{split}
m_0(r,\varphi\circ f)&=\frac{1}{2\pi}\int\limits_{|z|=r}\log\bigl (1+|\varphi\circ f|^2\bigl )d\theta+\frac{1}{2\pi}\int\limits_{|z|=\frac{1}{r}}\log\bigl (1+|\varphi\circ f|^2\bigl )d\theta +O(1)\\
&\le \frac{1}{2\pi}\int\limits_{|z|=r}\log\frac{1}{\|\sigma_2\circ f\|_2}d\theta+\frac{1}{2\pi}\int\limits_{|z|=\frac{1}{r}}\log\frac{1}{\|\sigma_2\circ f\|_2}d\theta +O(1)\\
&=m_{0,f}(r,\div^\infty(\varphi)).
\end{split}
\end{align}
On the other hand, we have
\begin{align}\label{3.8}
\begin{split}
N_0(r,\div^{\infty}(\varphi\circ f))&=N_0(r,\div (\sigma_2\circ f))\\
&=T_{0,f}(r,c_1(\div^\infty(\varphi)))-m_{0,f}(r,\div^\infty(\varphi))+O(1).
\end{split}
\end{align}
From (\ref{3.7}), (\ref{3.8}) and by Lemma \ref{3.2}, we have
$$ T_{0}(r,\varphi\circ f)\le T_{0,f}(r,c_1(\div^\infty(\varphi)))+O(1)= O(T_{0,f}(r,\Omega)).$$
Combining the above inequality with (\ref{3.6}), we get
\begin{align}\label{3.9}
\biggl\|\ m_0(r,\xi_2)\le O\left(\log^+T_{0,f}(r,\Omega)+\log^+\frac{1}{R_0-r}\right).
\end{align}

From (\ref{3.4}), (\ref{3.5}) and (\ref{3.9}) we get
$$\biggl\|\ m_0(r,\xi) \le O(\log^+T_{0,f}(r,\Omega))+O\left(\log^+\frac{1}{R_0-r}\right).$$
The lemma is proved. 
\end{proof}

\section{Proof of Second main theorem for holomorphic curves}

We have the following lemma from \cite{Q12} (which is a special case of \cite[Lemma 3.14]{NWY}).
\begin{lemma}[{see \cite[Lemma 4.4]{Q12}}]\label{10.1}
Let $M$ be a semi-Abelian variety and let $D$ be an algebraic divisor on $M$ such that $\St(D)=\{0\}$.
Then there exists a smooth equivariant compactification $\overline M$ of $M$ such that the closure $\overline D$ of $D$ in $\overline M$ is big, generally positioned.
\end{lemma}

\vskip0.2cm
\noindent
\textbf{(c)  Proof of Theorem \ref{1.2}}

Let $\overline M$ be a smooth equivariant compactification of $M$ which is chosen as in Lemma \ref{10.1}. We may regard $f$ as a holomorphic curve into 
$\overline M$.

As in the Section $\S 2$, denote by $\mathcal M_{\overline M}^{*}$ the sheaf of germs of meromorphic functions on $\overline M$ which do not identically vanish, and the sheaf $\mathcal U^{1}_{\overline M}$ is defined by
$$\begin{array}{ccccc}
0\rightarrow \C^{*}\rightarrow &\mathcal M^{*}_{\overline M}&\overset{d\log}{\rightarrow}&\mathcal U^{1}_{\overline M} & \rightarrow 0.\\
&\gamma&\mapsto &d\log\gamma & \\
\end{array}$$
Since $M$ is a semi-Abelian variety, by taking the standard coordinates from the universal cover $\C^n\rightarrow M$ of $M$ which gives automatically sections $\omega_i$ in $H^{0}({\overline M}, \mathcal U^1_{\overline M})$, we may assume that
 $\omega^{i}\in H^{0}({\overline M}, \mathcal U^1_{\overline M})$ for all $1\le i\le n$.

We define functions $\xi^i$ by setting $f^*\omega^i=\xi^i dz$. Then by Lemma \ref{1.1} we have
$$
\biggl\|\ m(r,\xi^i )\le O(\log^+ T_{0,f}(r,\Omega ))+O\left(\log^+ \frac{1}{R_0-r}\right),\ \forall 1\le i\le n.
$$
If $D$ is a divisor on $M$ such that $\overline D$ is generally positioned in $\overline M$, by Theorem \ref{3.2} we obtain
\begin{align}\label{9.5}
\biggl\|\ m(r,\xi^i )\le O(\log^+ T_{0,f}(r,c_1(\overline D) ))+O\left(\log^+ \frac{1}{R_0-r}\right),\ \forall 1\le i\le n.
\end{align}

\noindent
\begin{proof}[Proof of theorem \ref{1.2}]

 Without loss of generality we assume that $D$  is irreducible. If $\mathrm{St}(D)\not=\{0\}$, by taking the quotient $q: M\rightarrow M/\mathrm{St}(D)$ and deal with the holomorphic curve  $q\circ f:\A(R_0)\rightarrow M/\mathrm{St}(D)$ and the divisor $D/\mathrm{St}(D)$, then we may reduce to the case where $D$ is
 irreducible and $\mathrm{St}(D)=\{0\}$. Thus we may assume that $D$ is irreducible and $\mathrm{St}(D)=\{0\}$. 

By Lemma \ref{10.1}, there exists a smooth equivariant compactification $\overline M$ of $M$, in which $\overline D$ is generally positioned and big.

Let $J_k(f):\A(R_0)\rightarrow J_k(\overline M, \log \partial M)\cong\overline M\times \C^{nk}$ be the $k$-jet lifting of $f$. We have the following claim from \cite[Claim 6.1]{Q12}

\begin{claim}\label{new4.53} There exists a number $k_0$ such that 
\begin{align*}
\pi_2\big (J_{k_0}(\overline D, \log \partial M)\big )\cap \pi_2(\overline{J_{k_0}f(\A(R_0) )}^{\mathrm{Zar}})\not= \pi_2(\overline{J_{k_0}f(\A(R_0) )}^{\mathrm{Zar}}),
\end{align*}
where $\overline{J_{k_0}f(\A(R_0) )}^{\mathrm{Zar}}$ is the Zariski closure of $J_{k_0}f(\A(R_0) )$ in $J_k(\overline M, \log \partial M)$.
\end{claim}
Since the proof of this claim is just follow the proof of \cite[Claim 6.1]{Q12} with the same lines, we will omit its proof here.

Let $\{U_\alpha\}$ be an affine open covering of $\overline M$
such that
 \begin{equation}
  \label{tri}
   L(\overline{D})|_{U_\alpha} \cong U_\alpha \times \C.
   \end{equation}
We take $\sigma \in H^{0}(\overline M,L(\overline D))$ so that $\div (\sigma )=\overline D$ and take local holomorphic functions $\sigma_\alpha
=\sigma|_{U_\alpha}$ given by the trivialization \eqref{tri}.

We fix a Hermitian metric $\| \cdot \|$ in $L(\overline D)$ and choose positive smooth functions $h_\alpha$ on $U_\alpha$ such that
$$\dfrac{1}{\|\sigma (x)\|}=\dfrac{h_\alpha(x)}{|\sigma_\alpha(x)|}, \qquad x\in U_\alpha.$$

By Claim \ref{new4.53}, there exists a polynomial $R(w )$ in variable 
$$w=(w_{lk})\in \pi_2\left(\overline {J_{k_0}(f)(\A(R_0) )}^{\text{Zar}}\right)\cong \C^{nk_0}$$
such that
\begin{align*}
\pi_2\left (J_{k_0}(\overline D,\log\partial M)\right )\cap\pi_2\left
 (\overline {J_{k_0}(f)(\A(R_0) )}^{\text{Zar}}\right )&\subset\left\{w \in\pi_2\left (\overline {J_{k_0}(f)(\A(R_0) )}^{\text{Zar}}\right )|\ R(w )=0\right\}\\
&\not=\pi_2\left (\overline {J_{k_0}(f)(\A(R_0) )}^{\text{Zar}}\right ).
\end{align*}
Then we have an equation on every $U_\alpha\times \pi_2\left (\overline {J_{k_0}(f)(\A(R_0) )}^{\text{Zar}}\right )$ of the form:
\begin{align}{\label{new4.31}}
b_{\alpha 0}\sigma_\alpha +b_{\alpha 1}d\sigma_\alpha +\cdots +b_{\alpha k_0}d^{k_0}\sigma_\alpha=R(w ),
\end{align}
where $b_{\alpha i}$ are jet differentials on $U_\alpha$.
 Therefore, in every $U_j\times \pi_2\left (\overline {J_{k_0}(f)(\A(R_0) )}^{\text{Zar}}\right )$, we have
$$\dfrac{1}{\|\sigma \|}=\dfrac{1}{|R|}
\dfrac{h_\alpha}{|\sigma_\alpha|}=\dfrac{1}{|R|}
\left| h_\alpha b_{\alpha 0}+h_\alpha b_{\alpha 1}\dfrac{d\sigma_\alpha}{\sigma_\alpha}+\cdots +h_\alpha b_{\alpha k_0}\dfrac{d^{k_0}\sigma_\alpha}{\sigma_\alpha}\right|.$$
Choose relatively compact open subsets $U'_\alpha$ of $U_\alpha$ so that
 $\cup_\alpha U'_\alpha =\overline M$. For every $\alpha$, there exist
 positive constants $C_\alpha$ such that 
 $$h_\alpha|b_{\alpha i}|\le\sum_{\text{finite}}h_\alpha |b_{\alpha
 ilk\beta_{lk}(x)}|\cdot |w_{lk}|^{\beta_{lk}}\le
 C_\alpha\sum_{\text{finite}}|w_{lk}|^{\beta_{lk}}$$
for all $x\in U'_\alpha$.
 
By making $C_\alpha$ larger if necessary, there exists $d_\alpha >0$ such that for $f(z)\in U'_\alpha$ we have
$$h_\alpha (f(z))|b_{\alpha i}(J_{k_0}(f)(z))|\le C_\alpha
\left(1+\sum_{\underset{1\le k\le k_0}{1\le l\le n}}\left| w_{lk}\circ J_{k_0}(f)(z)\right|\right)^{d_\alpha}.$$
Setting $\xi_{(k)}=(w_{1k}(J_{k}(f)),\cdots ,w_{2k}(J_{k}(f)))$ and
$\xi_{l(k)}=w_{lk}(J_{k}(f))$, for $f(z) \in U'_\alpha$ we have
$$\dfrac{1}{\|\sigma (f(z))\|}\le\dfrac{1}{|R(\xi_{(1)}(z),\cdots
,\xi_{(k_0)}(z))|}\sum_{j=1}^{N}C_\alpha \left(1+\sum_{\underset{1\le k\le
k_0}{1\le l\le n}}|
\xi_{l (k)} (z)|\right)^{d_\alpha}$$
$$\quad\quad\quad\quad\times  \left (
1+\left| \dfrac{d\sigma_\alpha}{\sigma_\alpha}(J_1(f)(z))\right| + \cdots +
\left| \dfrac{d^{k_0}\sigma_\alpha}{\sigma_\alpha}\left(J_{k_0}(f)(z)\right)\right|
 \right ).$$
This implies that
\begin{align}\nonumber
m_{0,f}(r;\overline D)&=\dfrac{1}{2\pi}\int_{0}^{2\pi}\log^{+}\dfrac{1}{\|\sigma (f(re^{i\theta})\|}d\theta +\dfrac{1}{2\pi}\int_{0}^{2\pi}\log^{+}\dfrac{1}{\|\sigma (f(\frac{1}{r}e^{i\theta})\|}d\theta +O(1)\\
\label{new4.28}
&\le m_0\left (r,\dfrac{1}{R(\xi_{(1)}(z),\cdots ,\xi_{(k_0)}(z))}\right )\\
\nonumber
&+O\biggl (\sum_{\underset{1\le k\le k_0}{1\le l\le n}}\frac{1}{2\pi}\int_{0}^{2\pi}\bigl (\log^+|\xi_{l(k)}(re^{i\theta })|+\log^+|\xi_{l(k)}(\frac{1}{r}e^{i\theta })|\bigl )d\theta\\
\nonumber
&+\sum_{\underset{1\le k\le k_0}{1\le \le n}}\frac{1}{2\pi}\int_{0}^{2\pi}\bigl (\log^+\bigl |\dfrac{d^k\sigma_\alpha}{\sigma_\alpha}(J_k(f)(re^{i\theta })\bigl |+\log^+\bigl |\dfrac{d^k\sigma_\alpha}{\sigma_\alpha}(J_k(f)+(\frac{1}{r}e^{i\theta })\bigl )\bigl |d\theta\biggl )+O(1).
\end{align}
By Lemma \ref{1.1} and by (\ref{9.5}) we have
\begin{align}\label{new4.29}
\biggl\| \ m_0\left (r,\dfrac{d^k\sigma_j\circ f}{\sigma_j\circ f}\right )=O(\log^+T_{0,f}(r;c_1(\overline D)))+O\left (\log^+\frac{1}{R_0-r}\right ).
\end{align}
Combining Lemma \ref{1.1}, (\ref{new4.28}) and (\ref{new4.29}) we obtain
\begin{align}\label{4.4}
\biggl\| \quad m_{0,f}(r;\overline D)=O(\log^+T_{0,f}(r;c_1(\overline D)))+O\left(\log\frac{1}{R_0-r}\right).
\end{align}

We next estimate the counting function $N_0(r, f^*D)$.
We see that for all $z\in\A(R_0)$,
$$\ord_zf^{*}D>k\Leftrightarrow J_k(f)(z)\in J_k(\overline D,\log\partial M).$$
Then, from (\ref{new4.31}) we have
$$\ord_zf^*D-\min\{\ord_zf^*D,k_0\}\le\ord_z\div_0(R(\xi_{(1)},\cdots ,
\xi_{(k_0)})).$$
Thus 
\begin{align}\label{new4.33}
N_0(r,f^*D) - N^{[k_0]}_0(r,f^*D)\le N_0(r,\div_0R(\xi_{(1)},\cdots ,\xi_{(k_0)})).
\end{align}
By the first main theorem and by Lemma \ref {1.1}, we have
\begin{align}
\nonumber
\|\ N_0(r,\div_0R(\xi_{(1)},\cdots ,\xi_{(k_0)}))&\le  T_0(r,R(\xi_{(1)},\cdots , \xi_{(k_0)}))+O(1)\\
\nonumber
&\le O\left(\sum_{\underset{1\le k\le k_0}{1\le l\le n}}T_0(r,\xi_{l(k)})\right )+O(1)\\
\nonumber
&=O\left(\sum_{\underset{0\le k\le k_0-1}{1\le l\le n}}m_0(r, \xi_{l(k)})\right )+O(1)\\
\label{new4.34}
&=O(\log^+T_{0,f}(r;c_1(\overline D)))+O\left(\log^+\frac{1}{R_0-r}\right).
\end{align}
By (\ref {4.4}), we have
\begin{align}\label{new4.32}
T_{0,f}(r;c_1(\overline D))\le N_0(r,f^*D) +O(1).
\end{align} 

Combining (\ref{new4.33}), (\ref{new4.34}) and (\ref{new4.32}), we obtain
$$\biggl\|\ T_{0,f}(r;c_1(\overline D))\le N^{[k_0]}(r,f^*D)+O(\log^+T_{0,f}(r;c_1(\overline D)))+O\left(\log^+\frac{1}{R_0-r}\right).$$
The theorem is proved.
\end{proof}

\end{document}